\documentclass[12pt,reqno,twoside]{amsart}
\usepackage{geometry}

\usepackage[english]{babel}

\usepackage[T1]{fontenc}

\usepackage[lf]{Baskervaldx} 
\usepackage[bitstream-charter]{mathdesign}

\usepackage{a4wide}

\usepackage{latexsym}
\usepackage{tensor}
\usepackage{times,amsmath,amsthm}
\usepackage{graphicx}
\usepackage[all]{xy}
\usepackage{paralist}
\usepackage{subfigure}
\usepackage{multicol}
\usepackage{comment}

\usepackage{xcolor}
\definecolor{Chocolat}{rgb}{0.36, 0.2, 0.09}
\definecolor{BleuTresFonce}{rgb}{0.215, 0.215, 0.86}
\usepackage[colorlinks,final,hyperindex,pagebackref=true]{hyperref}
\hypersetup{
    linkcolor=blue,
	citecolor=purple,
	filecolor=BleuTresFonce,
	urlcolor=BleuTresFonce,
	pdftitle={On Proper Descent of Aspherical Affine Surfaces},
	pdfauthor={Buddhadev Hajra},
	pdfsubject={This paper discusses certain characterizations of smooth affine surfaces that are homotopic to real spheres (including circles).},
	pdfkeywords={Rational homotopy, elliptic homotopy type, Eilenberg-MacLane space, Stein space, open algebraic surface, logarithmic Kodaira dimension}
}
\usepackage[capitalise]{cleveref}
\usepackage{placeins}

\usepackage{tikz, tkz-tab, tikz-cd}
\usepackage{mathtools}

\usepackage[new]{old-arrows}

\newtheorem{theorem}{Theorem}[section]
\newtheorem{prop}[theorem]{Proposition}
\newtheorem{lemma}[theorem]{Lemma}
\newtheorem{cor}[theorem]{Corollary}
\newtheorem{question}[theorem]{Question}

\newtheorem{mainthm}{Theorem}

\theoremstyle{definition}
\newtheorem{definition}[theorem]{Definition}
\newtheorem{remark}[theorem]{Remark}
\newtheorem{example}[theorem]{Example}

\numberwithin{equation}{section}

\newcommand{\A}{\mathbb{A}}
\newcommand{\C}{\mathbb{C}}

\newcommand{\G}{\mathbb{G}}
\newcommand{\N}{\mathbb{N}}
\newcommand{\Q}{\mathbb{Q}}

\newcommand{\Z}{\mathbb{Z}}

\newcommand{\Kbar}{\bar{\kappa}}

\newcommand{\PP}{\mathbb{P}}

\newcommand{\wt}{\widetilde}

\newcommand{\Spec}{\operatorname{Spec}}

\newcommand{\Pic}{\operatorname{Pic}}

\newcommand{\rank}{\operatorname{rank}}
\newcommand{\mult}{\operatorname{mult}}

\newcommand{\codim}{\operatorname{codim}}

\newcommand{\st}{\,:\,}

\usetikzlibrary{decorations.markings}
\tikzstyle{vertex}=[circle, draw, inner sep=0pt, minimum size=5pt]

\begin{document}

\title[On Proper Descent of Smooth Affine Surfaces with Finite Homotopy Rank-Sum]{On Proper Descent of Smooth Affine Surfaces\\with Finite Homotopy Rank-Sum}
\author[B. Hajra]{Buddhadev Hajra}

\address{Stat-Math-Unit, Indian Statistical Institute Kolkata, 203 B. T. Road, Baranagar, Kolkata 700108, India}

\email{hajrabuddhadev92@gmail.com}



\begin{abstract}
We study the descent behaviour of homotopy-theoretic properties of smooth complex affine surfaces under finite surjective morphisms. We first examine the {\it Eilenberg--MacLane property} and show, by means of an explicit counterexample, that it does not descend under proper morphisms in general. This negative result motivates the introduction of a weaker notion, the {\it finite homotopy rank-sum property}. Our main theorem establishes that this property does descend under proper morphisms between smooth affine surfaces of logarithmic Kodaira dimension at most zero. The proof relies essentially on the recent classification of smooth complex affine surfaces of log non-general type characterized by these two properties. As a further application, we classify smooth affine surfaces properly dominated by the complex algebraic 2-torus, thereby clarifying an earlier remark of M. Furushima.
\end{abstract}

\maketitle

\tableofcontents

\section{\bf Introduction}\label{se1}

We begin by motivating our interest in studying smooth complex algebraic surfaces that are also Eilenberg--MacLane spaces. A path-connected topological space $X$ is said to be an \emph{Eilenberg--MacLane space} if, for some integer $n \in \Z_{>0}$ and a group $G$,
\[
\pi_k(X)=
\begin{cases}
G, & \text{for } k=n,\\[2mm]
0, & \text{otherwise},
\end{cases}
\]
and is denoted by $K(G,n)$. For any abelian group $G$ and any integer $n\,\in\, \N$, in the category of CW-complexes an Eilenberg--MacLane $K(G,\,n)$ space exists
and it is unique up to homotopy (cf. \cite[Chapter 8, \S1, Corollary 5]{Spa1966} and \cite[Corollary 2]{MT1968}). So we use the same notation $K(G,\,n)$ to denote any CW complex belonging to this homotopy class.

\medskip

Throughout this paper, all algebraic varieties and morphisms are defined over the field of complex numbers $\C$. An affine variety $V$ is said to be \emph{properly dominated} by an affine variety $W$ if there is a proper surjective morphism $\varphi: V \to W$ of $V$ onto $W$. For affine varieties, properness is equivalent to finiteness, hence proper domination amounts to the existence of a finite surjective morphism. In commutative algebraic terms, writing $V=\Spec B$ and $W=\Spec A$, this is equivalent to $B$ being integral (or finite) over $A$. 

By Noether normalization, every affine variety of dimension $d$ admits a finite surjective morphism onto $\mathbb{A}^d$; in particular, every affine variety is properly dominated by the affine space of the same dimension.

\medskip

Suppose $f \colon X \to Y$ is a finite surjective morphism between smooth algebraic surfaces. A classical result of J.-P.~Serre asserts that the image of the induced homomorphism 
$f_{\ast} \colon \pi_1(X) \to \pi_1(Y)$ 
has finite index in $\pi_1(Y)$, and moreover, this index divides the degree of~$f$ (cf.~\cite{Ser1959}). Consequently, if $\pi_1(X)$ is finite, then $\pi_1(Y)$ must also be finite. This naturally leads to the following question.

\begin{question}\label{Ques : Basic Ques}
Let $f \colon X \to Y$ be a finite surjective morphism of smooth affine surfaces. If $\pi_2(X)$ is finite, is $\pi_2(Y)$ also finite?
\end{question}

To approach this, consider a smooth affine surface $V$ for which $\pi_2(V)$ is finite. It is well known that every algebraic variety admits a universal covering space. Let $p \colon \widetilde{V} \to V$ denote the universal covering map. By the Hurewicz theorem,\footnote{
For any path-connected space $X$ and positive integer $n$, there exists a group homomorphism 
$h_{*} \colon \pi_{n}(X)\to H_{n}(X;\Z)$, called the \emph{Hurewicz homomorphism}.\\[2mm]
\textbf{Hurewicz Theorem.} \emph{If $n\ge 2$ and $X$ is $(n-1)$-connected, then $h_{*} \colon \pi_{n}(X)\to H_{n}(X;\Z)$ is an isomorphism, and $h_{*}\colon \pi_{n+1}(X)\to H_{n+1}(X;\Z)$ is surjective.} \\[2mm]
For this, we refer to \cite[Chapter 8, Theorem 3]{MT1968} and \cite[Exercise 23, Section 4.2]{Hat2002}.}we have $\pi_2(\widetilde{V}) \cong H_2(\widetilde{V};\Z)$. Since $p$ is a covering map, it follows that $\pi_i(V)\cong \pi_i(\widetilde{V})$ for all $i\ge 2$. 

As $V$ is a smooth affine surface, $\widetilde{V}$ is a complex Stein manifold of dimension $2$. By a classical result of A.~Andreotti and T.~Frankel, $H_2(\widetilde{V};\Z)$ is a free abelian group and $H_k(\widetilde{V};\Z)=0$ for all $k>2$ (cf.~\cite{AF1959}). Therefore, the finiteness of $\pi_2(V)$ implies $\pi_2(V)=(1)$. Repeated application of the Hurewicz theorem and the Andreotti--Frankel result further yields $\pi_i(V)=(1)$ for all $i \ge 2$. Hence, $V$ is a $K(G,1)$-space with $G=\pi_1(V)$.

\medskip

Thus, Question~\ref{Ques : Basic Ques} motivates the following broader question, originally posed to the author by Prof.~R.~V.~Gurjar.

\begin{question}\label{Main Ques: Descent of EM}
Let $f \colon X \to Y$ be a finite surjective morphism of smooth complex affine surfaces. If $X$ is an Eilenberg--MacLane space, is $Y$ also an Eilenberg--MacLane space?
\end{question}

The answer to this question is, in general, negative. Shortly after it was posed, R.~V.~Gurjar, S.~R.~Gurjar, and the author constructed a counterexample, which appears at the end of this paper (see Example~\ref{Ex: Non-example to the main question}). To investigate this problem further, one requires explicit examples of smooth affine Eilenberg--MacLane surfaces in order to test whether the property descends under additional hypotheses. This leads naturally to the following question:

\begin{center}
\emph{Which smooth affine surfaces are Eilenberg--MacLane spaces?}
\end{center}

\medskip

In~\cite{GGH2023}, the author, together with R.~V.~Gurjar and S.~R.~Gurjar, obtained a classification of smooth affine Eilenberg--MacLane surfaces of \emph{non-general type}, i.e.~those with logarithmic Kodaira dimension $\Kbar(X) \le 1$, since the case $\Kbar(X)=2$ (surfaces of general type) remains not fully understood.

\begin{theorem}[cf. {\cite[Result 2.1]{GGH2023}}]\label{Thm: Classification of EM surfaces}
Let $X$ be a non-contractible smooth complex affine $K(G,1)$-surface. Then the following hold:
\begin{enumerate}[\indent\rm(1)]
\item If $\Kbar(X)=-\infty$, then $X$ is an $\A^1$-bundle over a smooth algebraic curve not isomorphic to $\A^1$ or $\PP^1$.
\item If $\Kbar(X)=0$, then either $X\cong \C^*\times \C^*$ or $X \cong H[-1,0,-1]$\footnote{The surface $H[-1,0,-1]$ was originally defined by T. Fujita in his famous paper \cite{Fuj1982}. For detailed description, see \cite{GGH2023}.}.\\Moreover, if $X \cong H[-1,0,-1]$, there exists a $2$-fold unramified Galois cover of $X$ isomorphic to $\C^* \times \C^*$.
\item If $\Kbar(X)=1$, then $X$ admits a $\C^*$-fibration whose reduced fibers are all isomorphic to $\C^*$. Consequently, there exists a finite unramified Galois cover of $X$ that is a $\C^*$-bundle over a smooth algebraic curve not isomorphic to $\PP^1$.
\end{enumerate}
\end{theorem}

In Example~\ref{Ex: Non-example to the main question}, we will see that an Eilenberg--MacLane surface may descend under a finite surjective morphism onto a smooth affine surface homotopic to the real $2$-sphere. Nevertheless, in this paper, we identify several significant situations in which the answer to Question~\ref{Main Ques: Descent of EM} is affirmative.

\medskip

Recently, in~\cite{BH2025}, the authors established the following related result.

\begin{theorem}[cf. {\cite[Theorem G]{BH2025}}]\label{Thm: Classification of affine surfaces with finite homotopy rank-sum}
Let $X$ be a smooth complex affine surface of log non-general type satisfying the finite homotopy rank-sum property, i.e.~$X$ is such that $$\sum_{i\ge 2}\rank\pi_i(X):=\sum_{i\ge 2}\dim_{\Q}\left(\pi_i(X)\otimes_{\Q}\Q\right)$$ is finite. Then $X$ is one of the following:
\begin{enumerate}[\indent\rm(1)]
\item An Eilenberg--MacLane $K(\pi,1)$-space;
\item A smooth $\Q$-homology plane with $\pi_1(X)\cong \Z/2\Z$;
\item A surface homotopic to $S^2$.
\end{enumerate}
\end{theorem}

In view of Theorem~\ref{Thm: Classification of affine surfaces with finite homotopy rank-sum} and Example~\ref{Ex: Non-example to the main question}, we are led to the following question.

\begin{question}\label{Main Ques: Descent of finite homotopy rank-sum}
Let $f \colon X \to Y$ be a finite surjective morphism of smooth complex affine surfaces. If $X$ satisfies the finite homotopy rank-sum property, does $Y$ also satisfy it?
\end{question}

In the present work, we answer Question~\ref{Main Ques: Descent of finite homotopy rank-sum} affirmatively in the case where $\Kbar(X)\le 0$. The general case remains open.

\medskip

We prove the following theorem showing how in certain situations Question \ref{Main Ques: Descent of EM} may have an affirmative answer.

\begin{mainthm}[{Theorem \ref{Thm: Descent of A^1 bundle over affine curve}}]\label{thm:mainA}
Let $X$ be a Zariski-locally trivial $\A^1$-bundle over a smooth affine irreducible curve. Suppose $f \colon X \to Y$ is a finite surjective morphism onto a smooth affine surface. Then $Y$ is a Zariski-locally trivial $\A^1$-bundle over a smooth affine irreducible curve.
\end{mainthm}

This result extends an earlier theorem of M.~Furushima, who classified all normal affine surfaces properly dominated by $\C\times \C^*$ (see the Main Theorem in~\cite{Fur1989}).

\medskip

We also obtain the following classification results for smooth affine surfaces properly dominated by $\C^*\times \C^*$.

\begin{mainthm}[{Theorem \ref{Thm: Affine surfaces with kappa=0 properly dominated by 2-torus}, \ref{Thm: Affine surfaces with kappa negative properly dominated by 2-torus}}]\label{thm:mainB}
Let $Y$ be a smooth affine surface properly dominated by $\C^*\times \C^*$.
\begin{enumerate}[\indent\rm(1)]
\item If $\Kbar(Y)=-\infty$, then $Y$ is isomorphic to either $\A^2$ or $\A^1\times \C^*$.

\medskip

\item If $\Kbar(Y)=0$, then $Y$ is one of the following: 

\medskip

    \begin{enumerate}[\indent]
        \item[\rm(2a)] $\C^*\times \C^*$;
        \item[\rm(2b)] Fujita’s surface $H[-1,0,-1]$;
        \item[\rm(2c)] A surface homotopic to $S^2$ that belongs to $\mathcal{S}_0$---the class defined in \cite{JSXZ2024};
        \item[\rm(2d)] A smooth $\Q$-homology plane with $\pi_1(Y)\cong \Z/2\Z$.
    \end{enumerate}
\end{enumerate}
\end{mainthm}

\medskip

Finally, we prove the following theorem, which provides an affirmative answer to Question~\ref{Main Ques: Descent of finite homotopy rank-sum} in the case $\Kbar(X)\leq 0$, as desired.

\begin{mainthm}[{Theorem \ref{Main Theorem - kappa bar negative case}, \ref{Main Theorem - kappa bar = 0 case}}]\label{thm:mainC}
Let $f \colon X \to Y$ be a finite surjective morphism of smooth affine surfaces. If $\Kbar(X)\leq0$ and $X$ satisfies the finite homotopy rank-sum property, then so does $Y$.
\end{mainthm}

\section{\bf Preliminaries}
{\bf [A] --} Firstly, we will list certain notations and conventions that we follow throughout this article.

\subsection{Notations \& Conventions}
\begin{enumerate}[\indent$\bullet$]

\item If $X$ is a simplicial complex or a finite CW complex, then we denote by 
$\pi_i(X)$, $H_i(X; G)$ (resp. $H^i(X; G)$), and $\wt{H}_i(X; \Z)$ (for $i \in \Z_{\ge 0}$) 
the $i$-th homotopy group, the $i$-th homology (resp. cohomology) group with coefficients in $G$, 
and the $i$-th reduced integral homology group of $X$, respectively.

\item By the phrase \emph{higher homotopy groups} of $X$, we mean the groups $\pi_i(X)$ for all $i \ge 2$.

\item[$\bullet$] Throughout, we work with smooth algebraic surfaces and their morphisms defined over the field of complex numbers~$\C$.

\item The $n$-dimensional affine and projective spaces over $\C$ are denoted by $\A^n$ and $\PP^n$, respectively.

\item For any integer $N \ge 0$, let $\C^{N*}$ denote the complex affine line with $N$ points removed. 
When $N = 0$, we have $\C^{0*} = \C$. 
For $N = 1$ and $N = 2$, we write $\C^*$ and $\C^{**}$, respectively, instead of $\C^{1*}$ and $\C^{2*}$.

\item For a smooth non-complete surface $S$, we use the following notation:
\[
\Kbar(S) : \text{the logarithmic Kodaira dimension of } S.
\]
For the definition of $\Kbar$, see~\cite{Iit1982}.

\item For a smooth complex algebraic variety $X$, we adopt the following notation:
\[
b_i(X) \; (i \in \Z_{>0}) : \text{the $i$-th Betti number } \dim_{\C} H_i(X, \C),
\]
\[
e(X) : \text{the Euler--Poincaré characteristic of } X, \text{ defined by } 
e(X) = \sum_{i = 0}^{\infty} (-1)^i b_i(X).
\]

\item For a smooth complex affine variety $X$, we reserve the following notation:
\[
\rho(X) := \rank \Pic(X) = \dim_\Q \left(\Pic(X) \otimes_\Z \Q\right),
\]
\[
\gamma(X) := \rank \left( \Gamma(X, \mathcal{O}_X)^\ast / \C^\ast \right),
\]
where $\Gamma(X, \mathcal{O}_X)^\ast / \C^\ast$ is a free abelian group of finite rank.
\end{enumerate}

\medskip

{\bf [B] --} We will use the following well-known results in our latter proofs.
	
\begin{lemma}[J.P. Serre, cf. \cite{Ser1959}]\label{Lem: surjective induced map at the level of pi_1}
	Let $\varphi:X \to Y$ be a dominant morphism of normal, irreducible algebraic varieties over $\C$. Then the following are true.
	\begin{enumerate}
		\item[\rm (1)] Image of the induced map $\varphi_*:\pi_1(X) \to \pi_1(Y)$ has finite index in $\pi_1(Y)$. In fact, if $\varphi$ is quasi-finite the index $[\pi_1(Y):\varphi_*(\pi_1(X))]$ divides $\deg \varphi$.
		\item[\rm (2)] If the general fiber of $\varphi$ is irreducible, then $\varphi_*$ is onto.
	\end{enumerate}
	Consequently, if $X$ is a normal algebraic variety and $Z$ is a proper closed subvariety of $X$, then the natural homomorphism $\pi_1(X-Z) \to \pi_1(X)$ is surjective. Moreover, if $X$ is smooth and $\codim_X Z \geq 2$, then this homomorphism is an isomorphism.
\end{lemma}

The following result is implicitly used in our latter proofs. The proof of this result follows from covering space theory.

\begin{lemma}\label{Lem: Pull-back of connected covering}
	Let $f: (Z,z_0) \to (W,w_0)$ be a continuous surjective map between connected topological manifolds such that the induced homomorphism $f_\ast:\pi_1(Z,z_0) \to \pi_1(W,w_0)$ is onto. Then for any connected topological covering $p:(\wt{W}, \wt{w}_0) \to (W,w_0)$, the induced covering $\wt{p}:(\wt{Z}, \wt{z}_0) \to (Z,z_0)$ is also connected, where $\wt{Z}:=Z\times_{W}\wt{W}$.
\end{lemma}

\medskip

{\bf [C] --} We will now recall some well-known results which we will often use in our later proofs.
\subsection{Basics about Eilenberg-MacLane smooth complex algebraic varieties}

The following results about an Eilenberg-MacLane algebraic variety (or a finite CW complex) are well-known. We will implicitely use these results in our later proofs.
\begin{lemma}\label{Lem: Well-known about EM}
	The following assertions hold.
	
	\begin{enumerate}[\indent\rm(1)]
		\item Let $X$ be a complex algebraic variety which is an Eilenberg-MacLane space. Then $X$ is $K(G,1)$, for $G=\pi_1(X)$. Moreover, if $X$ is simply connected, then it is contractible.
		
		\item For any smooth $K(G,1)$ quasi-projective variety, the group $G$ is torsion-free.
		
		\item Let $X, Y$ be connected topological spaces and $f: X \to Y$ a fiber bundle with connected fiber $F$. If $Y$ is a $K(G,1)$ space and $F$ is a $K(G',1)$ space with $G=\pi_1(Y)$ and $G'=\pi_1(F)$ then $X$ is a $K(G'',1)$ space for $G''=\pi_1(X)$.
		
		\item No smooth complex algebraic $K(G,1)$ surface can admit a $\mathscr{C}^\infty$-smooth fiber bundle structure over $\PP^1$.	
	\end{enumerate}
\end{lemma}

\subsection{A formula of M. Suzuki}
M. Suzuki proved an important formula for the Euler-Poincar{\'e} characteristic of a smooth affine surface fibered over a smooth algebraic curve (cf. \cite{Suz1977}).
\begin{theorem}[{{\bf Suzuki's Formula}; cf. \cite{Suz1977}}]\label{Suzuki's formula}
	Let $f:V\to C$ be a surjective morphism from a smooth affine surface $V$ onto a smooth algebraic curve $C$ such that a general fiber $F$ of $f$ is irreducible. Let $p_1,\ldots,p_n$ be all the points in $C$ such that $f$ is not $C^{\infty}$-locally trivial in a neighbourhood of $p_i\;(1\leq i \leq n)$. Then
	$$e(V)=e(C)\cdot e(F) + \sum\limits_{i=1}^{n}\left(e(F_i)-e(F)\right),$$ where $F_i:=f^{-1}(p_i)$ for $1\leq i \leq n$. 
	
	Further, every difference in any bracket on the right-hand side of this equality is non-negative. If $e(F_i)=e(F)$ for some $i$, then $F$ is isomorphic to either $\C$, or $\C^*$ and $(F_i)_{red}$ is isomorphic to $F$.
\end{theorem}
The proof of this uses the theory of plurisubharmonic functions. An algebro-geometric proof was given by the first author (cf. \cite{Gur1997}).

\subsection{Nori's lemma and Xiao Gang's generalization}

M. V. Nori proved, in \cite{Nor1983}, an important result on a short exact sequence involving fundamental groups of smooth algebraic varieties.

\begin{theorem}[{{\bf Nori's lemma}; \cite[Lemma 1.5]{Nor1983}}]\label{Nori's Lemma}
Let $X$ and $Y$ be smooth connected complex algebraic varieties and $f\,:\, X
\,\longrightarrow\, Y$ a dominant morphism with a connected general fiber $F$. Assume that there is a closed subset
$S\, \subset\, Y$, of codimension at least two, outside which all the fibers of $f$ have at least one smooth point (i.e., $f^{-1}(p)$ is generically reduced on at least
one irreducible component of $f^{-1}(p)$ for all $p \,\in\, Y\setminus S$). 
Then the natural exact sequence of fundamental groups $$\pi_1(F)\,\longrightarrow\, \pi_1(X) \,\longrightarrow\,\pi_1(Y)\,\longrightarrow\, (1)$$
is exact.
\end{theorem}

See also \cite[Remark 3.13, 3.14]{GGH2023}, for the two important remarks related to the above lemma. These are implicitly used in the latter proofs.\\

The following result due to Xiao Gang is a useful generalization of Theorem \ref{Nori's Lemma}.

\begin{lemma}[{cf. \cite[Lemma 2]{Gan1991}}]\label{Lem: Gang's Generalization}
Let $f \,:\, X\, \longrightarrow\, C$ be a surjective morphism from a smooth algebraic surface to a smooth algebraic curve
such that a general fiber $F$ of $f$ is smooth and irreducible. Let $\{p_1,\,\cdots,\,p_r\}\, \subset\, C$ be the image of the
multiple fibers of $f$ with the multiplicity of the fiber over $p_i$ being $m_i\, >\, 1$, $1\, \leq\, i\, \leq\, r$.
Let $\wt{C}$ be the smooth completion of $C$. Let $\wt{C}\setminus C\,=\{p_{r+1},\,\cdots,\, p_{r+\ell}\}$ and
$g\, =\, {\rm genus}(\wt{C})$. Then there is an exact sequence
$$\pi_1(F)\, \longrightarrow\, \pi_1(X ) \, \longrightarrow\, \Gamma \, \longrightarrow\, (1),$$
where $\Gamma$ is the group with generators $$\alpha_1,\, \beta_1, \,\cdots, \,\alpha_g,\, \beta_g,\, \gamma_1,\,\cdots,
\, \gamma_r,\,\gamma_{r+1},\, \cdots,\, \gamma_{r+\ell}$$ and relations
$$[\alpha_1,\, \beta_1]\cdots [\alpha_g, \,\beta_g] \gamma_1\cdots \gamma_r \cdot \gamma_{r +1}\cdots \gamma_{r +\ell}
\,=\, 1 \,= \,\gamma_1^{m_1}\,=\, \cdots \,=\, \gamma_r^{m_r};$$
here $[a, \,b]\,=\, aba^{-1}b^{-1}$. If there are no multiple fibers in the $F$-fibration $f$, then
$\Gamma$ is the fundamental group of $C$.
\end{lemma}

\subsection{Ramified Covering Trick}
We first recall the definition of the total multiplicity of a fiber of a fibration.
\begin{definition}
	Let $\varphi: X \to C$ be a fibration on a smooth, connected surface $X$ to a smooth algebraic curve $C$. Let $F$ be a fiber of $\varphi$ having irreducible components $F_1, \ldots, F_r$ for $r\in \Z_{>0}$ with multiplicities $m_1,\ldots, m_r$ respectively. Thus scheme-theoretically $F$ can be written as, $$F=m_1F_1+m_2F_2+\cdots+m_rF_r.$$
	The \emph{total multiplicity} of $F$ is defined by $\mult(F):=\gcd(m_1,\ldots, m_r)$. The fiber $F$ is said to be a \emph{multiple fiber} if $\mult(F)>1$.
\end{definition}
The solution of Fenchel's Conjecture, called ``Fenchel's Lemma'' is useful for the theorem that follows about the ramified covering trick in elimination of multiple fibers of an algebraic fibration. We refer \cite{Cha1983} for the solution of this conjecture. 

\begin{theorem}[{{\bf Fenchel's Lemma}; cf. \cite{Cha1983}}]\label{Fenchel's Conjecture}
	Let $C$ be a smooth irreducible quasi-projective curve of genus $g$. Let $p_1, p_2, \ldots, p_r$ be distinct points in $C$ and $m_1, m_2, \ldots, m_r$ be arbitrary integers $>1$, where we assume that if $C \cong \PP^1$ then $r\geq 2$ and if further $r=2$ then $m_1=m_2$. Then there is a finite Galois covering $\pi : D \to C$ which is branched precisely over $p_1, p_2, \ldots, p_r$ and the ramification index over $p_i$ is $m_i$ for each $i$.
\end{theorem}

\begin{theorem}[{{\bf Ramified Covering Trick}; cf. \cite[III, 3.2.1]{Miy2001}}]\label{Ramified covering trick}
	Let $f:X \to C$ be a surjective morphism with an irreducible general fiber $F$, where $X$ is a smooth irreducible algebraic surface and $C$ is a smooth irreducible quasi-projective curve. Let $p_1, p_2, \ldots, p_r$ be the points in $C$ such that $f^{\ast}(p_i)$ has multiplicity $m_i>1$ and for any $p\in C$ other than any $p_i$ the total multiplicity of the fiber $f^{\ast}(p)$ is $1$. We assume that if $C \cong \PP^1$ then $r\geq 2$ and in case $r=2$ we have $m_1=m_2$. Let $\pi: D \to C$ be as in Fenchel's Lemma \ref{Fenchel's Conjecture}. Then the following diagram commutes,
	$$
	\begin{tikzcd}[row sep=large, column sep=large]
	\overline{X \times_{C} D} \arrow[r, "p_1"] \arrow[d, "p_2"'] &X \arrow[d, "f"]\\
	D \arrow[r, "\pi"'] & C
	\end{tikzcd}
	$$ and the following assertions hold.
	\begin{enumerate}[\indent\rm(1)]
		\item The normalized fiber product  is $\overline{X \times_{C} D}$ is smooth.
		\item $p_1 : \overline{X \times_{C} D} \to X$ is a finite Galois {\'e}tale covering of $X$.
		\item $p_2 : \overline{X \times_{C} D} \to D$ has no multiple fibers.
	\end{enumerate}
\end{theorem}

\medskip

{\bf [D] --} The following standard results will be used in subsequent arguments. 
For the sake of completeness, we include brief sketches of their proofs.

\begin{lemma}\label{Lem: Homotopy equivalence if the fiber of a bundle is contractible}
Let $X$ and $Y$ be connected CW complexes, and let $f \colon X \to Y$ be a fiber bundle with connected fiber $F$. If $F$ is contractible, then $f$ is a homotopy equivalence.
\end{lemma}

\begin{proof}
The long exact sequence in homotopy associated with the fiber sequence 
\[
F \xhookrightarrow{i} X \xrightarrow{f} Y
\]
is given by
\[
\cdots \to \pi_k(F) \xrightarrow{i_\ast} \pi_k(X) \xrightarrow{f_\ast} \pi_k(Y) \xrightarrow{\delta_k} \pi_{k-1}(F) \to \cdots \qquad (k \ge 1).
\]
Since $F$ is connected and contractible, we have $\pi_i(F) = 0$ for all $i \ge 0$. Hence, exactness implies that 
\[
f_\ast \colon \pi_i(X) \xrightarrow{\ \cong\ } \pi_i(Y)
\]
is an isomorphism for every $i$. The result now follows from Whitehead’s theorem, which asserts that a map between connected CW complexes inducing isomorphisms on all homotopy groups is a homotopy equivalence (see, e.g., \cite[Theorem~4.5]{Whi1949b}).
\end{proof}

\begin{cor}\label{Cor: b_2 of an A^1-bundle}
For a smooth complex algebraic variety $X$, let $f \colon X \to B$ be an algebraic $\A^1$-bundle over a smooth complex algebraic curve $B$. Then $b_2(X) = 0$ if and only if $B$ is affine.
\end{cor}

\begin{proof}
Since $\A^1$ is contractible, Lemma~\ref{Lem: Homotopy equivalence if the fiber of a bundle is contractible} implies that 
\[
f \colon X \to B
\]
is a homotopy equivalence. In particular, the induced map
\[
f_\ast \colon H_2(X; \Z) \xrightarrow{\ \cong\ } H_2(B; \Z)
\]
is an isomorphism, so $b_2(X) = b_2(B)$. Therefore, $b_2(X) = 0$ if and only if $b_2(B) = 0$, which happens precisely when $B$ is non-complete; equivalently, $B$ is affine.
\end{proof}

\begin{lemma}\label{Lem: Q-factorial fibration}
Let $f: X \to B$ be an $F$-fibration from a smooth affine surface $X$ onto a smooth algebraic curve $B$, where $F$ is a smooth affine rational curve. If $X$ is $\Q$-factorial, then the following hold:
\begin{enumerate}[\indent\rm(1)]
    \item $B$ is affine.
    \item All fibers of $f$ are irreducible. 
\end{enumerate}
\end{lemma}

\begin{proof}
Let $f : X \to B$ be an $F$-fibration as in the statement, with $X$ a smooth affine $\Q$-factorial surface and $B$ a smooth algebraic curve. Embed $X$ as an open subset of a smooth projective surface $V$ such that $f$ extends to a $\PP^1$-fibration 
\[
\varphi : V \to \bar{B}
\]
onto a smooth projective curve $\bar{B}$, where $B \subset \bar{B}$ is the smooth completion.

Since $X$ is $\Q$-factorial, the Picard group $\Pic(V)$ is finitely generated, hence $b_1(V)=0$. The general fiber of $\varphi$ is isomorphic to $\PP^1$, which is connected, and therefore the induced map $\varphi_\ast : \pi_1(V) \to \pi_1(\bar{B})$ and hence
\[
\varphi_\ast : H_1(V; \Z) \to H_1(\bar{B}; \Z)
\]
is surjective. Consequently, $b_1(\bar{B}) \le b_1(V) = 0$, implying $\bar{B} \cong \PP^1$. Thus $B$ is a smooth rational curve.

Since $\bar{B} \cong \PP^1$, one general fiber $F$ of $\varphi$, together with a cross-section, say $S$ of $\varphi$ and all but one irreducible component of every singular fiber, freely generate $\Pic(V)$.  

If $B$ is projective, then $\bar{B} \cong \PP^1$. Then a similar argument shows that a general fiber of $f$, together with all but one irreducible component of each singular fiber of $f$, freely generates $\Pic(X)$. Hence $\rho(X) \ge 1$, and equality holds if and only if $f$ has no reducible fibers. Therefore, $\Pic(X)$ cannot be torsion, contradicting the $\Q$-factoriality of $X$. This contradiction shows that $B$ must be affine, proving (1).

Now assume that $B$ is a smooth affine rational curve. By the same reasoning as above, all but one irreducible component of each singular fiber of $f$ freely generate $\Pic(X)$. Hence $\rho(X) \ge 1$ whenever $f$ has a reducible fiber, again contradicting the $\Q$-factoriality of $X$. Therefore, all fibers of $f$ are irreducible, establishing (2).
\end{proof}

Now we will quote the following two useful results from \cite{JSXZ2024}. 

\begin{lemma}[{cf. \cite[Lemma 2.6, Corollary 2.7]{JSXZ2024}}]
Let $X$ be a smooth quasi-projective surface and $\pi : X \to B$ a $\G_m$-bundle over a smooth algebraic curve $B$. Then we have:
\begin{enumerate}[\indent\rm(1)]
    \item There is a finite {\'e}tale cover $B'\to B$ of degree $\leq 2$ inducing a finite {\'e}tale cover $X' := X \times_{B} B' \to X$ and a $\G_m$-bundle $X' \to B'$ which is untwisted.
    \item The log Kodaira dimensions satisfy $$\Kbar(X') = \Kbar(X) = \Kbar(B) = \Kbar(B').$$
    \item Assume that $\pi$ is untwisted and $B$ is a smooth rational affine curve. Then $\pi$ is a trivial $\G_m$-bundle.
    \item If $X$ be a smooth affine surface and $\pi$ is a $\G_m$-bundle, then $B$ is affine.
\end{enumerate}
\end{lemma}

\begin{lemma}[cf. {\cite[Lemma 2.14]{JSXZ2024}}]\label{Lem: A^1-bundle over affine rational curve is trivial}
	Any $\A^1$ bundle over a smooth affine rational curve $B$ is trivial.
\end{lemma}

The following lemma is very useful in the context of proper (finite) descent of smooth complex algebraic varieties. 
\begin{lemma}\label{Lem: Finite morphism basic properties}
Let $f \colon V \to W$ be a finite surjective morphism between smooth irreducible complex algebraic varieties. Then the following hold:
\begin{enumerate}[\indent \rm(1)]
    \item $\Kbar(V) \ge \Kbar(W)$. Moreover, if $f$ is {\'e}tale, then $\Kbar(V) = \Kbar(W)$ \textup{(cf.~\cite{Iit1982})}.
    
    \item The induced homomorphism 
    \[
    f_{*} \colon H_{i}(V; \Q) \longrightarrow H_{i}(W; \Q)
    \]
    is surjective for all $i$. In particular, $b_i(V) \ge b_i(W)$ for all $i$ \textup{(cf.~\cite{Gie1964})}.
    
    \item There is an induced injective homomorphism
    \[
    f^{*}\otimes Id_{\Q} \colon \Pic(W) \otimes_{\Z} \Q \longrightarrow \Pic(V) \otimes_{\Z} \Q.
    \]
    In particular, $\rho(W) \le \rho(V)$ \textup{(follows from the Projection Formula)}.
\end{enumerate}
\end{lemma}

\section{\bf Proof of Main Results}

The results in this section aim to clarify this descent phenomenon by establishing precise geometric conditions under which it holds, starting from the case of negative log Kodaira dimension and proceeding to more general settings. More precisely, we address Question~\ref{Main Ques: Descent of EM} and Question~\ref{Main Ques: Descent of finite homotopy rank-sum}, which concern the descent of the Eilenberg--MacLane property and the finite homotopy rank-sum property under finite morphisms between smooth affine varieties. Our approach proceeds through a sequence of reductions and structural results, beginning with a key observation that serves as a common starting point for many of the subsequent proofs by allowing us to assume surjectivity at the level of fundamental groups.

\begin{remark}\label{Rem: pi_1 level map is surjective}
	We first reduce to the case where the natural homomorphism 
	\[
	f_\ast: \pi_1(X) \longrightarrow \pi_1(Y)
	\]
	is surjective. 
	
	If $f_\ast$ is not surjective, then by standard covering space theory, we can find a smooth affine variety $Z$ such that $f$ admits the following factorization:
	\[
	\begin{tikzcd}[column sep=normal, row sep=normal]
	X \arrow[rr, "f"] \arrow[dr, "h"'] & & Y\\
	& Z \arrow[ur, "g"'] & 
	\end{tikzcd}
	\]
	where $h : X \to Z$ is a finite morphism and $g : Z \to Y$ is a finite covering map satisfying $f = g \circ h$. Moreover, we have 
	\[
	\pi_1(Z) \cong g_\ast(\pi_1(Z)) = f_\ast(\pi_1(X)),
	\]
	so that $h_\ast : \pi_1(X) \to \pi_1(Z)$ is surjective, and $\deg g \cdot \deg h = \deg f$. 

    \medskip
	
	Now, if $Z$ is an Eilenberg--MacLane space, or if it has finite homotopy rank-sum, then the same property holds for $Y$, since a finite covering map $g$ preserves the higher homotopy groups of $Z$ and $Y$. Thus, we can reduce to the case where the induced homomorphism at the level of fundamental groups is surjective, by considering the finite surjective morphism $g$ in place of $f$ and establishing the result for $Z$ first.

\end{remark}

\subsection{On the descent of Eilenberg--MacLane affine surfaces when $\Kbar(X)=-\infty$}
We begin by proving the finite {\'e}tale descent is true for Zariski locally-trivial $\A^1$-bundles.

\begin{prop}\label{Prop: Etale cover of A^1-bundle preserves the bundle structure with the nature of base curve}
Let $V$ be a smooth affine surface admitting a Zariski locally trivial $\A^{1}$-bundle structure over a smooth affine (respectively, projective) curve. Let $\varphi : W \to V$ be a finite \'etale covering of smooth affine surfaces. Then $W$ also admits a Zariski locally trivial $\A^{1}$-bundle structure over a smooth affine (respectively, projective) curve.
\end{prop}

\begin{proof}
Assume first that $V$ admits a Zariski locally trivial $\A^{1}$-bundle structure over a smooth affine curve (or a projective curve of positive genus), say $B$. Then by Lemma~\ref{Lem: Homotopy equivalence if the fiber of a bundle is contractible}, the fiber being contractible implies that $V$ is homotopy equivalent to $B$, hence $V$ is a $K(G,1)$-space.

Since $\varphi : W \to V$ is a finite topological covering, $W$ is also a $K(G,1)$-space. Moreover, $\overline{\kappa}(V) = -\infty$ implies $\overline{\kappa}(W) = -\infty$. Therefore, by Theorem~\ref{Thm: Classification of EM surfaces}, the surface $W$ admits a Zariski locally trivial $\A^{1}$-bundle structure over a smooth algebraic curve $C\not\cong \PP^1$. Again, Lemma~\ref{Lem: Homotopy equivalence if the fiber of a bundle is contractible} yields that $W$ is homotopy equivalent to $C$.

\medskip
\noindent{\bf Case 1.} \emph{$B$ is affine.}

\medskip
Assume that $C$ is projective. Then either $C$ is an elliptic curve or $\mathrm{genus}(C)\geq 2$. In the latter case, $\pi_1(C)$ is not free, whereas $\pi_1(B)$ is free since $B$ is affine. The induced morphism $\varphi_\ast : \pi_1(W)\hookrightarrow \pi_1(V)$ is injective and $\pi_1(W)\cong \pi_1(C)$ is a subgroup of $\pi_1(V)\cong \pi_1(B)$, hence must be free. Thus $\mathrm{genus}(C)\geq 2$ is impossible, and therefore $C$ must be an elliptic curve.

In this situation
\[
e(W)=e(C)=0,\qquad 
e(W)=\deg(\varphi)\cdot e(V), \qquad
e(V)=e(B).
\]
Hence $e(B)=0$. Since $B$ is a smooth affine curve, it follows that $B\cong \C^\ast$ and $\pi_1(V)\cong \pi_1(\C^\ast)\cong \Z$. But $\pi_1(W)\cong \pi_1(C)\cong \Z\oplus\Z$, contradicting injectivity of $\varphi_\ast$. Thus $C$ cannot be projective. This concludes the affine case.

\medskip
\noindent{\bf Case 2.} \emph{$B$ is projective of positive genus.}

\medskip
Suppose $C$ is affine. Then $b_2(W)=b_2(C)=0$, while $b_2(V)=b_2(B)=1$. This contradicts Lemma~\ref{Lem: Finite morphism basic properties}\,(2), which prohibits such a finite surjective morphism. Therefore, if $B$ is projective of positive genus, $C$ must also be projective of positive genus.

\medskip
Finally, if $B\cong \PP^1$, then $V$ is simply connected. Hence, any finite \'etale covering $\varphi$ must be an isomorphism. This completes the proof.
\end{proof}

\medskip

Now we continue with the simplest geometric situation where the log Kodaira dimension of $X$ is $-\infty$. This case already captures the essential features of the descent problem and provides the basic mechanism behind our general arguments. The following theorem shows that in this case, the Eilenberg--MacLane property descends under any finite surjective morphism.\\

We now derive another important consequence of this theorem that shows how this phenomenon extends to a wider class of affine surfaces by moving from the special case of a trivial $\A^1$-bundle to a more general situation in which the total space itself may carry a nontrivial $\A^1$-bundle structure. The following result shows that the descent of the Eilenberg--MacLane property continues to hold in this broader setting.

\begin{theorem}\label{Thm: Descent of A^1 bundle over affine curve}
Let $X$ be a Zariski-locally trivial $\A^1$-bundle over a smooth affine irreducible curve. Suppose $f \colon X \to Y$ is a finite surjective morphism onto a smooth affine surface. Then $Y$ is a Zariski-locally trivial $\A^1$-bundle over a smooth affine irreducible curve.
\end{theorem}

\begin{proof}
With the notations in Remark~\ref{Rem: pi_1 level map is surjective}, we first factorize $f$ into $g$ and $h$.
By construction, the induced homomorphism $h_\ast : \pi_1(X) \to \pi_1(Z)$ is surjective. Since $X$ is an $\A^1$-bundle over a smooth affine curve, we have $b_2(X)=0$ by Corollary~\ref{Cor: b_2 of an A^1-bundle}. Using Lemma~\ref{Lem: Finite morphism basic properties}(2), we deduce $b_2(Z)=0$.

As $\Kbar(X) = -\infty$, then $\Kbar(Y) = \Kbar(Z) = -\infty$ using Lemma~\ref{Lem: Finite morphism basic properties}(3). Hence $Z$ admits an $\A^1$-fibration $\varphi : Y \to D$ over a smooth algebraic curve $D$. The induced map $\varphi_\ast : \pi_1(Z) \to \pi_1(D)$ is surjective, yielding $b_1(Z) \ge b_1(D)$.

Apply Suzuki’s formula (Lemma~\ref{Suzuki's formula}) to $\varphi$:
\[
e(Z) = e(D) + s,\qquad 
s := \sum_{i=1}^{r}\bigl(e(F_i)-e(\A^1)\bigr),
\]
where $\{F_i: 1\leq i\leq r\}$ are the singular fibers. Using $b_2(Z)=0$, this gives:
\[
b_1(D) - b_1(Z) = b_2(D) + s.
\]
Since $b_1(Z) \ge b_1(D)$, we obtain
\[
b_1(Z)=b_1(D), \qquad b_2(D)=0=s.
\]
Thus $D$ is affine and there is no reducible fiber, i.e., every singular fiber of $\varphi$, if there is any, is scheme-theoretically isomorphic to $m\A^1$ for an integer $m>1$ using \cite[Chapter I, Lemma 4.4]{Miy1981}).

If $\varphi$ has any multiple fiber, we eliminate all multiplicities via the ramified covering trick (cf. Lemma~\ref{Ramified covering trick}). Thus, we obtain a finite {\'e}tale cover $Z' \to Z$ such that $Z'$ admits an $\A^1$-fibration $\varphi' : Z' \to D'$ over an affine curve $D'$ having no multiple fiber. Since $h_\ast$ is surjective, we may pull back the {\'e}tale cover $Z'\to Z$ via $h$ to get a smooth affine surface $X'$ that is irreducible with a pulled back {\'e}tale cover $X'\to X$ admitting a finite surjective morphism $h' : X' \to Z'$. See the following commutative diagram:

\[\begin{tikzcd}[row sep=small, column sep=large]
	&&& {D'} \\
	{X':=X\times_{Z}Z'} && {Z':=\overline{Z\times_{D}D'}} \\
	&&& D \\
	X && Z
	\arrow["\begin{array}{c} \text{ramified}\\\text{covering} \end{array}", from=1-4, to=3-4]
	\arrow["{h'}", two heads, "\text{finite}"', from=2-1, to=2-3]
	\arrow["\text{finite {\'e}tale}"', from=2-1, to=4-1]
	\arrow["\varphi'"', "\text{no multiple fiber}", from=2-3, to=1-4]
	\arrow["\text{finite {\'e}tale}"', from=2-3, to=4-3]
	\arrow["h", two heads, "\text{finite}"', from=4-1, to=4-3]
	\arrow["\varphi"', from=4-3, to=3-4]
\end{tikzcd}\]

By Proposition~\ref{Prop: Etale cover of A^1-bundle preserves the bundle structure with the nature of base curve}, $X'$ is again an $\A^1$-bundle over a smooth affine curve---a surface that is structurally similar to $X$. Therefore, by what we proved above, $\varphi'$ must not have any reducible fiber. Thus, in turn, $\varphi$ has no multiple fibers as the ramified covering trick to eliminate multiple fiber ensures that the irreducible multiple fiber, say $\varphi^*(p)$, to split into reducible fiber with irreducible components being affine lines taken with reduced structures in $\varphi'^*(q)$ so that the greatest common divisor of these split multiplicities become $1$ and the fiber $\varphi'^*(q)$ become non-multiple, where $q$ is the unique pre-image of $p$ under the totally ramified cover $D'\to D$. This implies $\varphi$ is Zariski--locally trivial.

Thus $Y$ being a finite {\'e}tale descent of $Z$ is first of all topologically a $K(G,1)$-space and thus $Y$ admits a Zariski--locally trivial $\A^1$-bundle structure over a smooth algebraic curve, say $B$, using Theorem \ref{Thm: Classification of EM surfaces}. Now, $b_2(Z)=0$ implies that $b_2(Y)=0$ and hence $B$ is affine, as desired, by Lemma \ref{Lem: Finite morphism basic properties}(2) followed by Corollary \ref{Cor: b_2 of an A^1-bundle}. This completes the proof.
\end{proof}

\medskip

We deduce the following result as a consequence of Theorem \ref{Thm: Descent of A^1 bundle over affine curve}. This result is well-known and was proved in \cite{Miy1986} using the proof of the cancellation theorem for $\A^2$ and in \cite[Theorem 3]{GS1984} using the Mumford-Ramanujam method and Milnor's classification of finite groups acting freely on homotopy $3$-sphere.

\begin{cor}\label{Cor: Affine plane is the only smooth surface properly dominated by an affine plane}
   Let $Y$ be a smooth affine surface properly dominated by $\A^2$. Then $Y \cong \A^2$. 
\end{cor}
\begin{proof}
    Theorem \ref{Thm: Descent of A^1 bundle over affine curve} implies that $Y$ must be isomorphic to an $\A^1$-bundle over a smooth affine curve, say $C$. Thus, $Y$ is homotopy equivalent to $C$ and hence using Lemma \ref{Lem: Finite morphism basic properties}(2), we have $$0=b_1(\A^2)\geq b_1(Y)=b_1(C),$$ which in turn implies that $C\cong \A^1$ as $b_1(C)=0$ and $C$ is a smooth affine curve. Therefore, we conclude that $Y\cong \A^2$.
\end{proof}

Our next result shows that the product structure of $X$ with an $\A^1$-factor is preserved up to the number of punctures on the base curve whenever the base curve is affine and rational, extending a classical result of M.~Furushima (see \cite[Main Theorem, non-singular case]{Fur1989}).

\begin{cor}\label{Cor: Generalization of Furushima's result}
	Let $f: X \to Y$ be a finite surjective morphism of smooth affine surfaces with $X=\A^1 \times \C^{r*}$ for some $r \geq 0$. Then $Y \cong \A^1 \times \C^{s*}$ for some $0\leq s \leq r$.
\end{cor}

\begin{proof}
    Since $X=\A^1 \times \C^{r*}$ for some $r \geq 0$, it is a smooth rational surface. Thus $Y$ is an unirational surface and hence a rational surface by using Castelnuovo's rationality criterion. Using the notations from Theorem~\ref{Thm: Descent of A^1 bundle over affine curve}, we see that $Y$ is a locally trivial bundle over a smooth affine curve $D$. As $Y$ is rational, so is $D$, and hence $D \cong \C^{s*}$ for some $s \geq 0$. By Lemma~\ref{Lem: A^1-bundle over affine rational curve is trivial}, it follows that $Y$ is a Zariski-locally trivial $\A^1$-bundle over $\C^{s*}$, i.e., $Y \cong \A^1 \times \C^{s*}$.

    Finally, the inequality $r \geq s$ follows from $b_1(X) \geq b_1(Y)$.
\end{proof}

The preceding theorem provides a concrete case where the descent of the Eilenberg-MacLane property holds true.\\

\medskip

Before proving our main result in this section, we require the following theorem, which has recently been proved by the author in collaboration with A. Maharana and A. J. Parameswaran in a different project. This result will be published elsewhere. We state the result here without providing a proof.

\begin{theorem}\label{Thm: proper descent of homotopy sphere}
    Let $f\st X \to Y$ be a finite surjective morphism of complex smooth affine surfaces. Suppose that $X$ is homotopic to $S^2$. Then $Y$ is one of the following smooth affine surfaces:
    \begin{enumerate}[\indent\rm(1)]
        \item a topologically contractible surface;
        \item homotopic to $S^2$ again;
        \item a $\Q$-homology plane having fundamental group $\Z/2\Z$.
    \end{enumerate}
\end{theorem}

Next, we conclude this section with the proof of one of our key results.

\begin{theorem}\label{Main Theorem - kappa bar negative case}
	Let $f \colon X \to Y$ be a finite surjective morphism of smooth affine surfaces. If $\Kbar(X)=-\infty$ and $X$ satisfies the finite homotopy rank-sum property, then so does $Y$.
\end{theorem}
\begin{proof}
    First note that, since $Y$ is properly dominated by $X$ and $\Kbar(X)=-\infty$, thus $\Kbar(Y)=-\infty$ using Lemma \ref{Lem: Finite morphism basic properties}(3). Since $X$ satisfies the finite homotopy rank-sum property and $\Kbar(X)=-\infty$, it follows from Theorem \ref{Thm: Classification of affine surfaces with finite homotopy rank-sum} and Theorem \ref{Thm: Classification of EM surfaces} that $X$ is isomorphic to one of the following surfaces:
    \begin{enumerate}[\indent\rm(1)]
        \item a Zariski--locally trivial $\A^1$-bundle over a smooth algebraic curve not isomorphic to $\PP^1$;
        \item homotopic to $S^2$;
        \item a $\Q$-homology plane with fundamental group $\Z/2\Z$.
    \end{enumerate}
Note that the universal cover of a $\Q$-homology plane with fundamental group $\Z/2\Z$ is affine and topologically a Moore $M(\Z,2)$-space, i.e., homotopic to $S^2$. Therefore, without loss of generality, we can assume that $X$ is either a Zariski--locally trivial $\A^1$-bundle over a smooth algebraic curve not isomorphic to $\PP^1$ or is homotopic to $S^2$. We first consider the easy cases, which can be settled by directly applying our previous results.

\begin{enumerate}[\indent---]
    \item If $X$ is a Zariski--locally trivial $\A^1$-bundle over a smooth affine curve, then Theorem \ref{Thm: Descent of A^1 bundle over affine curve} implies that $Y$ is a $K(G,1)$-space. In particular, in this case, $Y$ satisfies the finite homotopy rank-sum property.
    
    \item If $X$ is homotopic to $S^2$, then using Theorem \ref{Thm: proper descent of homotopy sphere}, it turns out that $Y$ enjoys the finite homotopy rank-sum property.
\end{enumerate}

Now the following proposition completes the proof.
\end{proof}

\begin{prop}
    With the notations as in Theorem \ref{Main Theorem - kappa bar negative case}, suppose there exists a Zariski--locally trivial $\A^1$-bundle structure $\varphi: X\to C$ over a smooth projective curve $C$ with positive genus. Then $Y$ satisfies the finite homotopy rank-sum property.
\end{prop}

\begin{proof}
    Using Remark~\ref{Rem: pi_1 level map is surjective}, we can assume, without a loss of generality, that the induced homomorphism $f_\ast : \pi_1(X) \to \pi_1(Y)$ is surjective. Since $X$ is an $\A^1$-bundle over a smooth projective curve, we have $b_2(X)=1$ by the proof of Corollary~\ref{Cor: b_2 of an A^1-bundle}. Using Lemma~\ref{Lem: Finite morphism basic properties}(2), we deduce $b_2(Y)$ can be either $0$ or $1$. As $\Kbar(Y) = -\infty$, thus $Y$ admits an $\A^1$-fibration $\psi: Y \to D$ over a smooth algebraic curve $D$. Thus the induced homomorphism $\psi_\ast : \pi_1(Y) \to \pi_1(D)$ is surjective, yielding $b_1(Y) \ge b_1(D)$. Apply Suzuki’s formula (cf. Theorem~\ref{Suzuki's formula}) to $\varphi$:
    \[
    e(Y) = e(D) + s,\qquad 
    s := \sum_{i=1}^{r}\bigl(e(F_i)-e(\A^1)\bigr),
    \]
    where $\{F_i: 1\leq i\leq r\}$ are the singular fibers of $\psi$. This gives:
    \begin{equation}\label{eqn1}
        b_1(D) - b_1(Y) = s + b_2(D) - b_2(Y).
    \end{equation}

    \medskip
    \noindent{\bf Case 1.} \emph{Suppose $D$ is not isomorphic to $\A^1$ or $\PP^1$.}

    \medskip
    Since $\A^1$ cannot dominate any smooth algebraic curve except $\A^1$ and $\PP^1$, no fiber of $\varphi$ is mapped via $f$ horizontally with respect to the fibration $\psi$. Thus, $f$ is defined fiberwise. Take any regular fiber, say $F$ of $\psi$ over $q\in D$. Then, $f^{-1}(F)$ is non-empty as $f$ is surjective and is a disjoint union of fibers of $\varphi$ over $S:=\{p_1,\ldots,p_k\}\subset C$. Thus $f$ restricts to a finite surjective morphism $X\setminus f^{-1}(F) \to Y\setminus F$ of smooth affine surfaces such that $\varphi$ and $\psi$ restrict to $\A^1$-bundle over $C\setminus S$ and to an $\A^1$-fibration over $D\setminus \{q\}$ respectively. Note that, since $S\neq \emptyset$, thus $C\setminus S$ is a smooth affine curve and $D\setminus \{q\}$ is affine too. Now, Theorem \ref{Thm: Descent of A^1 bundle over affine curve} implies that $\psi$ restricts to a Zariski--locally trivial $\A^1$-bundle over $D\setminus \{q\}$. Hence, $\psi$ is also Zariski--locally trivial. Hence, $Y$ is a $K(G,1)$-space in this case.

    \medskip
    \noindent{\bf Case 2.} \emph{Suppose $D\cong \PP^1$.}

    \medskip
    \noindent
    Now, if $b_2(Y)=0$ then $e(Y)=1-b_1(Y)\leq 1$. But on the other hand, Suzuki's formula (cf. Theorem \ref{Suzuki's formula}) yields that $e(Y)\geq e(\PP^1)=2$, a contradiction. Thus, $b_2(Y)$ must be $1$ and hence $e(Y)=2-b_1(Y)$. Note that, equation \eqref{eqn1} implies that $s+b_1(Y)=0$ in this case which in turn implies that $b_1(Y)=0=s$, i.e., $e(Y)=2$ and $\psi$ has no reducible fiber. Hence, every singular fiber, if there is any, is scheme-theoretically isomorphic to $m\A^1$ for some integer $m>1$.  
    
    \medskip
    If $\psi$ has no singular fiber, then it is an $\A^1$-bundle over $\PP^1$, i.e., $Y$ is homotopic to $\PP^1$ using Lemma \ref{Lem: Homotopy equivalence if the fiber of a bundle is contractible} and $\PP^1$ is further homeomorphic to $S^2$.

    \medskip
    Suppose $\psi$ has either at least three multiple fibers, or exactly two multiple fibres with multiplicities $m_1, m_2$ such that $\gcd(m_1,m_2)>1$. We obtain a non-trivial finite {\'e}tale cover $Y'\to Y$ of degree $d>1$ (using the ramified covering trick in Theorem \ref{Ramified covering trick}). In the ramified covering process, we get an $\A^1$-fibration $\psi':Y'\to D'$ over a smooth projective curve. By pulling back the finite {\'e}tale cover $Y'\to Y$ we construct a finite {\'e}tale cover $X':=X\times_{Y}Y' \to X$ with a natural finite surjective morphism $X' \to Y'$. Note that, since $f_\ast$ is surjective, $X'$ is an irreducible smooth affine surface. By Proposition \ref{Prop: Etale cover of A^1-bundle preserves the bundle structure with the nature of base curve}, $X'$ is again an $\A^1$-bundle over a smooth projective curve with positive genus---a surface that is structurally similar to $X$. Therefore, if $D'$ is not isomorphic to $\PP^1$ then by Case 1 above, it follows that $\psi'$ is a Zariski--locally trivial $\A^1$-bundle over $D'$ and this would imply that $e(Y')=e(D')=2 - 2g(D')\leq 0$. On the other hand, $e(Y') = d \cdot e(Y) = 2d$. However, this is a contradiction as $d>1$. Hence $D'\cong \PP^1$ and thus by what we have proved above, we get $b_1(Y')=0$ and $b_2(Y')=1$ yielding $e(Y')=2$. However, once again, this is a contradiction as $e(Y')=2d$ and $d>1$.

    \medskip
    This concludes that $\varphi$ has at most two multiple fibres and, in the case of exactly two multiple fibers isomorphic to $m_1\A^1, m_2\A^1$, the two multiplicities $m_1,m_2$ must be coprime. In all of these cases, $Y$ turn out to be simply-connected using Lemma \ref{Lem: Gang's Generalization}. Furthemore, since $b_2(Y)=1$, we obtain $Y$ is homotopy equivalent to $S^2$, as observed earlier.

    \medskip
    \noindent{\bf Case 3.} \emph{Suppose $D\cong \A^1$.}

    \medskip
    We proceed with the proof, splitting into two subcases---$b_2(Y)$ being either $0$ or $1$.
    
    \medskip
    {\bf Subcase 3a.} \emph{Suppose $b_2(Y)=0$}

    \medskip
    Then equation \eqref{eqn1} again implies that $s+b_1(Y)=0$, whence $b_1(Y)=0=s$. Thus, $e(Y)=1$ and $\psi$ has no reducible fiber. Hence, every singular fiber, if there is any, is scheme-theoretically isomorphic to $m\A^1$ for some integer $m>1$. Note that $Y$ is a $\Q$-homology plane in this case. 

    \medskip
    Suppose $\psi$ has a multiple fiber. Then we obtain a non-trivial finite {\'e}tale cover $Y'\to Y$ of degree $d>1$ (using the ramified covering trick in Theorem \ref{Ramified covering trick}). In the ramified covering process, we get an $\A^1$-fibration $\psi':Y'\to D'$ over a smooth affine curve. By pulling back the finite {\'e}tale cover $Y'\to Y$ we construct a finite {\'e}tale cover $X':=X\times_{Y}Y' \to X$ with a natural finite surjective morphism $X' \to Y'$. Note that, since $f_\ast$ is surjective, $X'$ is an irreducible smooth affine surface. By Proposition \ref{Prop: Etale cover of A^1-bundle preserves the bundle structure with the nature of base curve}, $X'$ is again an $\A^1$-bundle over a smooth projective curve with positive genus---a surface that is structurally similar to $X$. Therefore, if $D'$ is not isomorphic to $\A^1$ then by {\sf Case 1} above, it follows that $\psi'$ is a Zariski--locally trivial $\A^1$-bundle over $D'$ and this would imply that $e(Y')=e(D')\leq 0$. On the other hand, $e(Y') = d \cdot e(Y) = d$. However, this is a contradiction as $d>1$. Hence $D'\cong \A^1$. Hence \cite[Remark 3.13, 3.14]{GGH2023} following Nori's lemma (cf. Theorem \ref{Nori's Lemma}) implies that $Y'$ is simply connected. Moreover, if $b_2(Y')=0$ then $e(Y')=1$, a contradiction as we noticed before that $e(Y')=d$ and $d>1$. This concludes that $\psi$ is a Zariski--locally trivial $\A^1$-bundle over $\A^1$, i.e., $Y\cong \A^2$.
    
    \medskip
    {\bf Subcase 3b.} \emph{Suppose $b_2(Y)=1$}

    \medskip
    Then equation \eqref{eqn1} implies that $s+b_1(Y)=1$, whence the tuple $(s, b_1(Y))$ takes values either $(0,1)$ or $(1,0)$. 
    
    \begin{enumerate}[\indent---]
        \item If $(s, b_1(Y))=(0,1)$, then $e(Y)=1$ in this case. Clearly, $\psi$ must have a singular fiber which is scheme-theoretically isomorphic to $m\A^1$ for an integer $m>1$, as $s=0$. For, if not, then $s=0$ would once again imply that $\psi$ is Zariski-locally trivial, whence $Y\cong \A^2$, a contradiction to the hypothesis that $b_1(Y)=1$. Hence, apply a similar argument as in {\sf Subcase 3a} involving the ramified covering trick applied to $\psi$, keeping the notations as in {\sf Subcase 3a}. Using Lemma \ref{Lem: Finite morphism basic properties}(2), we get $$1=b_2(X')\geq b_2(Y')\geq b_2(Y)=1,$$ whence $b_2(Y')=1$. If $D'\not\cong\A^1$, then by {\sf Case 1} above, it follows that $\psi': Y'\to D'$ is a Zariski--locally trivial $\A^1$-bundle and this would further imply that $b_2(Y')=0$, a contradiction. Hence $D'\cong \A^1$ and thus \cite[Remark 3.13, 3.14]{GGH2023} following Nori's lemma (cf. Theorem \ref{Nori's Lemma}) implies that $Y'$ is simply connected. But Lemma \ref{Lem: Finite morphism basic properties}(2) implies that $b_1(Y')\geq b_1(Y)$, a contradicion as $Y'$ is simply connected and $b_1(Y)=1$. Hence $(s, b_1(Y))=(0,1)$ never happens in this case.

        \medskip
        \item If $(s, b_1(Y))=(1,0)$, then $e(Y)=2$ in this case. Assume that $\psi$ has multiple fibers. Again, apply a similar argument as in {\sf Subcase 3a} involving the ramified covering trick applied to $\psi$, keeping the notations as in {\sf Subcase 3a}. As observed above, $b_2(Y')=1$ and thus $D'\cong \A^1$. Thus $Y'$ is simply connected as we noticed above, and hence $Y'$ is homotopic to $S^2$. Since $e(Y)=e(Y')=2$ with $e(Y')=d\cdot e(Y)$ and $d>1$, this leads to a contradiction. Hence $\psi$ has no multiple fiber. Therefore \cite[Remark 3.13, 3.14]{GGH2023} following Nori's lemma (cf. Theorem \ref{Nori's Lemma}) implies that $Y$ is simply connected. Hence, as shown previously, $Y$ is homotopic to $S^2$. 
    \end{enumerate}
    This completes the proof of the proposition.
\end{proof}

\subsection{On the descent of finite homotopy rank-sum if $\Kbar(X)=0$.}\mbox{}\\
Again, in this case, we start with the finite {\'e}tale descent case. In fact, the following more general statement is true.

\begin{prop}
    Let $S:=\{p_1,\ldots, p_k\}$ $(k\geq 0)$ be a set of $k$-poits in $\C^*\times \C^*$, where $S=\emptyset$ whenever $k=0$. Let $V$ be a smooth open algebraic surface with $\varphi: V\to \C^*\times \C^*\setminus S$ is a finite {\'e}tale covering. Then there exists a finite {\'e}tale covering $\Phi: \C^*\times \C^*\to \C^*\times \C^*$ such that $V\cong \C^*\times \C^*\setminus \Phi^{-1}(S)$.\\
    In particular, if $k=0$, then $V\cong \C^*\times \C^*$. 
\end{prop}
\begin{proof}
   The proof follows using the same idea of the proof of \cite[Lemma 3.3]{Fur1989}.  
\end{proof}

We first prove the following result, which will be very useful for our later discussion.

\begin{prop}\label{Prop: Affine surface properly dominated by 2-torus and with finite pi_1}
    Let $Y$ be a smooth affine surface properly dominated by $\C^*\times \C^*$ with $\Kbar(Y)=0$. If $\pi_1(Y)$ is finite, then $Y$ is one of the following smooth affine surfaces:
    \begin{enumerate}[\indent\rm(1)]
        \item homotopic to $S^2$;
        \item a $\Q$-homology plane with $\pi_1(Y)=\Z/2\Z$.
    \end{enumerate}
\end{prop}
\begin{proof}
    Let $X:=\C^* \times \C^*$ properly dominate $Y$ via a finite surjective morphism, say $f: X \to Y$. If $\deg(f)=1$, then $f$ being a surjective birational morphism turns out to be an isomorphism following Zariski's Main Theorem. Then $\pi_1(Y) \cong \pi_1(X) = \Z^2$, a contradiction to the hypothesis that $\pi_1(Y)$ is finite. This yields that $\deg(f)>1$. 
 
    With the notations in Remark~\ref{Rem: pi_1 level map is surjective}, we first factorize $f$ into $g$ and $h$. By construction, the induced homomorphism $h_\ast : \pi_1(X) \to \pi_1(Z)$ is surjective. Since $\pi_1(X)$ is finite, $\pi_1(Z)$ is finite too. Now consider the finite (equivalently, proper) universal covering map $p: \wt{Z} \to Z$ and we have the following commutative diagram:
	$$
	\begin{tikzcd}[row sep=large, column sep=large]
	\wt{X}:=X \times_{Z} \wt{Z} \arrow[r, "h'"] \arrow[d, "q"'] &\wt{Z} \arrow[d, "p"]\\
	X \arrow[r, "h"'] & Z
	\end{tikzcd}
	$$
	such that $\wt{X}$ is smooth as well as irreducible since $f_\ast: \pi_1(X) \to \pi_1(Y)$ is surjective and $q: \wt{X} \to X$ is a finite covering map which is essentially a pull back of the universal covering $p: \wt{Z} \to Z$ via $h: X \to Z$. Hence $h': \wt{X} \to \wt{Z}$ turns out to be a proper morphism with finite fibers, hence a finite surjective morphism. Since $X=\C^* \times \C^*$ and $q$ is a finite covering, $\wt{X}$ is isomorphic to $\C^* \times \C^*$. So $h'$ is again a finite surjective morphism similar to $h$. Since $\wt{X}$ being isomorphic to $\C^*\times \C^*$ is factorial and $h'$ is finite surjective, therefore using \cite{GMM2021}, it follows that $$\rho(\wt{Z})\leq \rho(\wt{X})=0,$$ whence $\rho(\wt{Z})=0$, i.e., $\wt{Z}$ is $\Q$-factorial. Then simply-connectedness of $\wt{Z}$ makes it factorial with $\Gamma(\wt{Z},\mathcal{O}_{\wt{Z}})^*=\C^*$. Also note that, since $p:\wt{Z} \to Z$ and $g:Z\to Y$ are finite {\'e}tale coverings, $$\Kbar(\wt{Z})=\Kbar(Z)=\Kbar(Y)=0.$$ This implies that $\wt{Z}\in \mathcal{S}_0$---the class defined in \cite{JSXZ2024}. Hence $\wt{Z}$ is homotopic to $S^2$ using \cite[Theorem 4.2]{JSXZ2024}. Now $2=e(\wt{Z})=\deg p\cdot e(Z)$ and hence the pair $(\deg p, e(Z))\in \N^2$ takes the value either $(2,1)$ or $(1,2)$. 

    Also note that, since $p:\wt{Z} \to Z$ and $g:Z\to Y$ both are finite {\'e}tale coverings, thus $\pi:=g\circ p:\wt{Z} \to Y$ is a universal covering map of degree $d=\deg p\cdot \deg g$. Now we consider the following two cases.\\
    
    {\bf Case 1.} \emph{Suppose $(\deg p, e(Z))=(2,1)$.}
    
    \medskip
    Since $1=e(Z)=\deg g\cdot e(Y)$, thus $e(Y)=1=\deg g$ and hence $Z=Y$, in this case, with $\pi:\wt{Z}\to Y$ an universal covering of degree $d=2$. Hence $\pi_1(Y)=\Z/2\Z$. Since $e(Y)=1$, it turns out that $H_2(Y;\Q)=(0)$. Also, $Y$ being a smooth affine surface, $H_i(Y;\Q)=(0)$ for all $i>2$ by \cite[Theorem 1]{AF1959}. Hence $Y$ is a $\Q$-homology plane with fundamental group $\Z/2\Z$.

    {\bf Case 2.} \emph{Suppose $(\deg p, e(Z))=(1,2)$.}
    
    \medskip
    Firstly, $\deg p=1$ implies that $\wt{Z}=Z$ and hence $\pi=g: Z \to Y$ is a universal covering. Since $2=e(Z)=\deg g\cdot e(Y)$, thus $(\deg g, e(Y))=(1,2) \text{ or } (2,1)$. 
    \begin{itemize}[\indent---]
        \item If $(\deg g, e(Y))=(1,2)$, then $Y=Z=\wt{Z}$ and hence $Y$ is homotopic to $S^2$ as $\wt{Z}$ was so.\\[-0.25 cm]
        \item If $(\deg g, e(Y))=(2,1)$, then hence, once again by the same argument, it follows that $Y$ is a $\Q$-homology plane with fundamental group $\Z/2\Z$.
    \end{itemize}

    Both cases together rest the proof.
\end{proof}

\begin{remark}
    In Proposition \label{Prop: Affine surface properly dominated by 2-torus and with finite pi_1}, we see that in the case of $Y$ being homotopy equivalent to $S^2$, it turns out that $Y\in \mathcal{S}_0$---the class defined in \cite{JSXZ2024}.
\end{remark}

\begin{prop}\label{Prop: No finite map from 2-torus to surface having infinite cyclic pi_1 and positive Euler characteristic}
	No smooth affine surface having a fundamental group isomorphic to $\Z$ and positive Euler characteristic can ever be properly dominated by $\C^* \times \C^*$.
\end{prop}
\begin{proof}
    If possible, assume that $X:=\C^* \times \C^*$ properly dominates a smooth affine surface $Y$ with $\pi_1(Y)\cong\Z$ and $e(Y)>0$ via a finite surjective morphism, say $f: X \to Y$.
 
    If $\deg(f)=1$, then $f$ is an isomorphism by Zariski's Main Theorem. Then $$\pi_1(Y) \cong \pi_1(X) = \Z^2,$$ a contradiction to the hypothesis that $\pi_1(Y)\cong \Z$. This yields that there is no such $f$ with $\deg(f)=1$. Choose an integer $d>1$, and assume that there exists no such $f$ of every degree $d' < d$.
    
    Assume that there is a finite surjective morphism $f: X \to Y$ of degree $d$ onto a smooth affine surface $Y$ with $\pi_1(Y)\cong \Z$ and $e(Y)>0$.\\
    
    {\bf Claim.} {\it This is possible only if the induced homomorphism $f_\ast: \pi_1(X) \to \pi_1(Y)$ is surjective.}
    \medskip
    
    {\it Proof of Claim.} --- The reason is as follows. Suppose $f_\ast: \pi_1(X) \to \pi_1(Y)$ is not surjective. 
 
   With the notations in Remark~\ref{Rem: pi_1 level map is surjective}, we first factorize $f$ into $g$ and $h$. By construction, the induced homomorphism $h_\ast : \pi_1(X) \to \pi_1(Z)$ is surjective, and since $f_\ast$ is not surjective, this factorization is non-trivial. This simply means that $\deg g >1$, yielding $\deg h < \deg f$. Note that $\pi_1(Z) \cong \Z$ as $[\pi_1(Y): f_\ast(\pi_1(X))] < \infty$ using Lemma \ref{Lem: surjective induced map at the level of pi_1} and $\pi_1(Y) \cong \Z$. Also, $e(Z)=\deg(g) \cdot e(Y) > 0$. But it follows from our induction hypothesis that there doesn't exist such an $h$, since $\deg h < \deg f = d$. This contradiction yields that if such an $f$ has to exist, $f_\ast: \pi_1(X) \to \pi_1(Y)$ must be surjective; hence the claim.\qed\\
    
    Consider a $2$-fold covering $Y'$ of $Y$. Then $Y'$ is a smooth affine surface and $\pi_1(Y') \cong \Z$. Therefore, $e(Y')=b_2(Y')$. Also, $e(Y')=2e(Y) \geq 2$. Therefore, $b_2(Y') \geq 2$. Since $f_\ast: \pi_1(X) \to \pi_1(Y)$ is surjective, $f$ induces a finite surjective morphism $f' : X' \to Y'$, where $X':=X \times_{Y} Y'$ and the morphism $X' \to X$, being a pullback of the Galois finite covering map $Y' \to Y$, is a finite unramified Galois covering. Therefore, $X' \cong \C^* \times \C^*$ again. But such an $f'$ cannot exist, since $b_2(X') =1 < 2 \leq b_2(Y')$. This completes the proof.
\end{proof}

As an immediate consequence of Proposition \ref{Prop: No finite map from 2-torus to surface having infinite cyclic pi_1 and positive Euler characteristic}, we can have the following result.

\begin{cor}\label{Cor: No finite map from 2-torus to surface having rank 1 abelian pi_1 and positive Euler characteristic}
	No smooth affine surface having an abelian fundamental group of rank $1$ and positive Euler characteristic can ever be properly dominated by $\C^* \times \C^*$.
\end{cor}

Now we are ready to prove one of our main results of this subsection.
\begin{theorem}\label{Thm: Affine surfaces with kappa=0 properly dominated by 2-torus}
	Let $Y$ be a smooth affine surface properly dominated by $\C^*\times \C^*$ with $\Kbar(Y)=0$. Then $Y$ is one of the following smooth affine surfaces:
    \begin{enumerate}[\indent\rm(1)]
        \item $\C^*\times \C^*$;
        \item Fujita's $H[-1,0,-1]$;
        \item homotopic to $S^2$, and moreover $Y\in \mathcal{S}_0$;
        \item a $\Q$-homology plane with $\pi_1(Y)=\Z/2\Z$.
    \end{enumerate}
\end{theorem}
\begin{proof}
    Let $X:=\C^* \times \C^*$ properly dominate $Y$ via a finite surjective morphism, say $f: X \to Y$. We start the same way as above. If $\deg(f)=1$, then $f$ is an isomorphism by Zariski's Main Theorem. So assume that $d:=\deg f >1$.

    With the notations in Remark~\ref{Rem: pi_1 level map is surjective}, we first factorize $f$ into $g$ and $h$. By construction, the induced homomorphism $h_\ast : \pi_1(X) \to \pi_1(Z)$ is surjective. 
    
    Let $Z'$ be a smooth strongly minimal surface such that $Z'$ is embedded in $Z$ as an affine open subspace. Then it is well-known that $\Kbar(Z)=\Kbar(Z')=0.$ Therefore $Z'$ is one of the surfaces listed in \cite[Table 1]{Koj1999}.\\
    
    {\bf Case 1.} \emph{Suppose that $\pi_1(Z)$ is finite.}
    
    \medskip
    Since $g: Z \to Y$ is a finite covering and hence $[\pi_1(Y):g_\ast(\pi_1(Z))]=\deg g$ is finite, therefore it follows that $\pi_1(Y)$ is finite too in this case. Thus, Proposition \ref{Prop: Affine surface properly dominated by 2-torus and with finite pi_1} implies that $Y$ is either homotopic to $S^2$ or a $\Q$-homology plane with $\pi_1(Y)=\Z/2\Z$.\\

    {\bf Case 2.} \emph{Suppose that $\pi_1(Z)$ is infinite.}

    \medskip
    Since $Z'\subset Z$ is a Zariski open subset, the homomorphism $\pi_1(Z') \to \pi_1(Z)$ induced by this open immersion must be surjective, and hence $\pi_1(Z')$ is infinite in this case. Thus, it follows from \cite[Table 1]{Koj1999} that $Z'$ is one of the five surfaces --- $O(4,1)$, $O(2,2)$, $O(1,1,1)$, $H[-1,0,-1]$ and $H[0,0]$. Now the rest of the proof of this case will be split into the following two claims.\\
    
    {\bf Claim 2a.} \emph{$Z'$ is isomorphic to either $O(1,1,1)$ or $H[-1,0,-1]$.}
    \medskip
    
    {\it Proof of Claim 2a.} --- Suppose $Z'$ is one of the three surfaces --- $O(4,1)$, $O(2,2)$ and $H[0,0]$, then $\pi_1(Z')\cong \Z$. Since $\pi_1(Z') \to \pi_1(Z)$ is surjective with $\pi_1(Z)$ an infinite group, hence $\pi_1(Z)\cong \Z$ whenever $Z'$ is one of $O(4,1)$, $O(2,2)$ and $H[0,0]$. Also, observe that $e(Z)\geq e(Z')$ as $$e(Z)= e(Z')+e(Z\setminus Z')$$ and $Z\setminus Z'$ is a disjoint union (possibly an empty union) of affine lines (cf. \cite[Lemma 1.4]{Koj1999}). Thus, $e(Z) > 0$ as $Z'$ being one of $O(4,1)$, $O(2,2)$ and $H[0,0]$ has positive Euler characteristic. However, this leads to a contradiction on the existence of such a finite morphism $h: X \to Z$; i.e, in other words, this contradicts the existence of such $Z$ being properly dominated by $\C^*\times \C^*$, following Proposition \ref{Prop: No finite map from 2-torus to surface having infinite cyclic pi_1 and positive Euler characteristic}. Therefore, it turns out that $Z'$ is either $O(1,1,1)$ (which is further isomorphic to $\C^*\times \C^*$) or $H[-1,0,-1]$, hence the claim.\qed\\
    
    {\bf Claim 2b.} \emph{$Z=Z'=O(1,1,1)$.}
    \medskip
    
    {\it Proof of Claim 2b.} --- Clearly, as we noticed earlier, $e(Z) \geq e(Z')=0$ and using \cite[Lemma 1.4]{Koj1999} it follows that $e(Z) = e(Z')$ if and only if $Z=Z'$. Suppose $e(Z)>0$. Since $h_\ast: \pi_1(X) \to \pi_1(Z)$ is surjective and $\pi_1(X)=\pi_1(\C^*\times \C^*)=\Z^2$, so $\pi_1(Z)$ is an abelian group and is isomorphic to one of the following:
    \begin{enumerate}
    	\item $\Z^2$;
    	\item $\Z \oplus (\Z/m\Z)$ for some $m > 1$;
    	\item $\Z$;
        \item $(\Z/m\Z)$ for some $m>1$;
    	\item $(\Z/m\Z) \oplus (\Z/n\Z)$ for some $m, n >1$;
        \item $(1)$ --- the trivial group.
    \end{enumerate}  
    In this case, since $\pi_1(Z)$ is infinite, the last three possibilities are ruled out and thus $\pi_1(Z)$ is isomorphic to either $\Z^2$, or $\Z \oplus (\Z/m\Z)$ for some $m > 1$, or $\Z$. Since we assumed that $e(Z)>0$, Proposition \ref{Prop: No finite map from 2-torus to surface having infinite cyclic pi_1 and positive Euler characteristic} and Corollary \ref{Cor: No finite map from 2-torus to surface having rank 1 abelian pi_1 and positive Euler characteristic} together imply that $\pi_1(Z)\cong \Z^2$. Now $\pi_1(Z)\cong \Z^2$ with $e(Z)>0$ yield that $b_2(Z) \geq 2$, which again contradicts the existence of such a finite morphism $h: X \to Z$ since $$1=b_2(\C^*\times \C^*)=b_2(X) < b_2(Z).$$ This implies that $e(Z)=0=e(Z')$ and hence $Z=Z'$. Already we have observed that $\pi_1(Z)$ must be abelian and since $\pi_1(H[-1,0,-1])$ is non-abelian (cf. \cite[Table 1]{Koj1999}), $Z'$ and hence $Z$ must be isomorphic $O(1,1,1)$, which is further isomorphic to $\C^*\times \C^*$. This completes the proof of the above claim.\qed\\
    		
    Clearly $Z$ being isomorphic to $\C^* \times \C^*$ is Eilenberg-MacLane. Hence, $Y$ is also Eilenberg-MacLane as $Z$ is a topological covering of $Y$. Now it follows from Theorem \ref{Thm: Classification of EM surfaces} that up to isomorphism of varieties, $\C^* \times \C^*$ and $H[-1,0,-1]$ are the only two affine smooth Eilenberg-MacLane surfaces with $\Kbar = 0$. Therefore, $Y$ is isomorphic to one of these two smooth affine surfaces. This completes the proof of this case and the theorem as well.
\end{proof}

\begin{remark}
In \cite{Fur1989}, M.~Furushima remarked (without any proof) at the end of his paper that $\C^* \times \C^*$ is the only smooth affine surface with $\Kbar = 0$ that can be properly dominated by $\C^* \times \C^*$. To the best of our knowledge, no detailed proof of this assertion has appeared in the literature since then. Our result in Theorem~\ref{Thm: Affine surfaces with kappa=0 properly dominated by 2-torus} provides a more comprehensive picture, showing that Furushima’s claim does not hold in full generality. Indeed, $\C^* \times \C^*$ occurs as a $2$-fold {\'e}tale finite cover of Fujita’s surface $H[-1,0,-1]$ (see \cite[Proposition~5.8]{GGH2023} and the proof therein). Moreover, our theorem opens the possibility that a smooth affine surface homotopic to $S^2$ could be properly dominated by $\C^* \times \C^*$. However, we do not presently know of any construction of a finite surjective morphism from $\C^* \times \C^*$ onto such a surface.
\end{remark}

\medskip

Another main result of this subsection is the following.

\begin{theorem}\label{Thm: Affine surfaces with kappa negative properly dominated by 2-torus}
   Let $Y$ be a smooth affine surface properly dominated by $\C^*\times \C^*$ with $\Kbar(Y)=-\infty$. Then $Y$ is isomorphic to either $\A^2$ or $\A^1\times \C^*$.
\end{theorem}

\begin{proof}
    Let $X:=\C^* \times \C^*$ properly dominate $Y$ via a finite surjective morphism, say $f: X \to Y$. We start the same way as above. If $\deg(f)=1$, then $f$ is an isomorphism by Zariski's Main Theorem, but this can not happen as $\Kbar(Y)=-\infty$ whereas $\Kbar(X)=0$. So, $\deg f>1$.

    With the notations in Remark~\ref{Rem: pi_1 level map is surjective}, we first factorize $f$ into $g$ and $h$. By construction, the induced homomorphism $h_\ast : \pi_1(X) \to \pi_1(Z)$ is surjective. Since $\pi_1(X)=\pi_1(\C^*\times \C^*)\cong \Z^2$, therefore $\pi_1(Z)$ must be abelian.
    
    Firstly, note that $\Kbar(Z)=\Kbar(Y)=-\infty$ as $g: Z \to Y$ is a finite {\'e}tale morphism. By a well-known result due to Fujita--Miyanishi--Sugie (cf. \cite{Fuj1979}, cf. \cite{MS1980}, cf. \cite{Sug1980}), $Z$ admits an $\A^1$-fibration, say $\varphi:Z \to B$ onto a smooth algebraic curve $B$. Since $Z$ is a smooth affine surface, every fiber of $\varphi$ is a disjoint union of affine lines, if taken with reduced structures (cf. \cite{Miy1981}). Since $X$ being $\C^*\times \C^*$ is a rational variety, $Y$ and $Z$ being dominated by $X$ become unirational and hence rational, as these are algebraic surfaces defined over $\C$ due to Castelnuovo's criterion for rationality. For the same reason, it turns out that the curve $B$ is also rational. Since any smooth affine surface properly dominated by $\C^*\times \C^*$ is $\Q$-factorial as observed before in the proof of Proposition \ref{Prop: Affine surface properly dominated by 2-torus and with finite pi_1}, $Z$ is $\Q$-factorial and hence Lemma \ref{Lem: Q-factorial fibration} implies that $B$ is a smooth affine curve and $\varphi$ has no reducible fibers. Since $\varphi$ is an $\A^1$-fibration, the induced homomorphism $\varphi_\ast: \pi_1(Z) \to \pi_1(B)$ is surjective. Since $\pi_1(Z)$ is abelian, so is $\pi_1(B)$. Therefore $B$ is a smooth affine rational curve with $\pi_1(B)$ abelian, whence $B$ is isomorphic to either $\A^1$, or $\C^*$.

    Let $\varphi$ have multiple fibers. Then these multiple fibers are scheme theoretically thick affine lines, i.e., isomorphic to $m\A^1$ for some integer $m>1$. Then, applying the ramified covering trick, we get a finite {\'e}tale covering $p:Z'\to Z$ such that $Z'$ is a smooth affine surface admitting an $\A^1$-fibration, say $\varphi': Z' \to C$ onto a smooth affine curve $C$. Also note that $\varphi'$ has no multiple fibers, but it has reducible fibers which are disjoint unions of affine lines taken with reduced structures. 
    Since $h_\ast: \pi_1(X) \to \pi_1(X)$ is surjective, we have the following commutative diagram:
    $$
        \begin{tikzcd}[row sep=large, column sep=large]
    	X':=X \times_{Z} Z' \arrow[r, "h'"] \arrow[d, "q"'] &Y' \arrow[d, "p"]\\
    	X \arrow[r, "h"'] & Y
    	\end{tikzcd}
    $$
    such that $X'$ is an irreducible, smooth affine surface and $q: X' \to X$ is a connected covering which is the pull back of the covering $p: Z' \to Z$ via $h: X \to Z$ and the induced map, say $h': X' \to Y'$ is a finite surjective morphism. Since $X=\C^* \times \C^*$ and $q$ is a finite covering, $X' \cong \C^* \times \C^*$. So $h'$ is again a finite surjective morphism similar to $h$. But since $Z'$ admits an $\A^1$-fibration $\varphi':Z' \to C$ which has reducible fibers, by Lemma \ref{Lem: Q-factorial fibration} we conclude that $Z'$ must not be $\Q$-factorial. But this contradicts the fact that $Z'$ is properly dominated by a factorial surface such as $\C^*\times \C^*$. This implies that $\varphi : Z \to B$ has no multiple fiber and all fibers are irreducible. So it turns out that $\varphi : Z \to B$ is an $\A^1$-bundle over $B$.\\
    
    {\bf Case 1.} \emph{Suppose $B \cong \A^1$.}

    \medskip
    It follows that $\varphi : Z \to B$ is an $\A^1$-bundle over $\A^1$ in this case and hence $Z\cong \A^2$. Since $g: Z \to Y$ is a finite {\'e}tale morphism, it follows that $Y\cong \A^2$.\\
    	
    {\bf Case 2.} \emph{Suppose $B \cong \C^*$.}

    \medskip	
    
    In this case, $\varphi : Z \to B$ is an $\A^1$-bundle over $\C^*$. Hence $Z \cong \A^1 \times \C^*$ following Lemma \ref{Lem: A^1-bundle over affine rational curve is trivial}. Since $g: Z \to Y$ is a finite {\'e}tale covering, it follows from Corollary \ref{Cor: Generalization of Furushima's result} that $Y\cong \A^1\times \C^*$.\\
    	
    This completes the proof.
\end{proof}

Now we are all set to prove our main result using all that we proved earlier.

\begin{theorem}\label{Main Theorem - kappa bar = 0 case}
	Let $f \colon X \to Y$ be a finite surjective morphism of smooth affine surfaces. If $\Kbar(X)=0$ and $X$ satisfies the finite homotopy rank-sum property, then so does $Y$.
\end{theorem}
\begin{proof}
    First note that, since $Y$ is properly dominated by $X$ and $\Kbar(X)=0$, thus $\Kbar(Y)$ is either $-\infty$ or $0$. Since $X$ satisfies the finite homotopy rank-sum property and $\Kbar(X)=0$, it follows from Theorem \ref{Thm: Classification of affine surfaces with finite homotopy rank-sum} and Theorem \ref{Thm: Classification of EM surfaces} that $X$ is isomorphic to one of the following surfaces:
    \begin{enumerate}[\indent\rm(1)]
        \item $\C^*\times \C^*$;
        \item Fujita's $H[-1,0,-1]$;
        \item homotopic to $S^2$;
        \item a $\Q$-homology plane with fundamental group $\Z/2\Z$.
    \end{enumerate}
Note that $\C^*\times \C^*$ is a $2$-fold {\'e}tale cover of $H[-1,0,-1]$. Also note that the universal cover of a $\Q$-homology plane with fundamental group $\Z/2\Z$ is homotopic to $S^2$. Therefore, without loss of generality, we can assume that $X$ is either $\C^*\times \C^*$ or is homotopic to $S^2$. 
 
\begin{enumerate}[\indent---]
    \item If $X\cong \C^*\times \C^*$, then Theorem \ref{Thm: Affine surfaces with kappa=0 properly dominated by 2-torus} and Theorem \ref{Thm: Affine surfaces with kappa negative properly dominated by 2-torus} together imply that $Y$ enjoys the finite homotopy rank-sum property.
    
    \item If $X$ is homotopic to $S^2$, then using Theorem \ref{Thm: proper descent of homotopy sphere}, it turns out that $Y$ enjoys the finite homotopy rank-sum property.
\end{enumerate}
This completes the proof.
\end{proof}

\section{\bf Examples}
We conclude this paper with an example demonstrating that, in general, the Eilenberg–MacLane descent under a finite surjective morphism of smooth affine surfaces does not hold.

\begin{example}\label{Ex: Non-example to the main question}
	Let $Y := (\PP^1 \times \PP^1) \setminus \Delta_{\PP^1 \times \PP^1}$, where $\Delta_{\PP^1 \times \PP^1}$ denotes the diagonal subvariety. Since $\Delta_{\PP^1 \times \PP^1}$ is an ample divisor in $\PP^1 \times \PP^1$, it turns out that $Y$ is a smooth affine surface. Also notice that $Y$ admits the structure of an $\A^1$-bundle over $B = \PP^1$. 
	
	Now consider a finite ramified morphism $C \to B$, where $C$ is a smooth non-rational curve, and let $X$ be the pullback $\A^1$-bundle over $C$. Then the induced morphism 
	\[
	f \colon X \longrightarrow Y
	\]
	is finite. Since $Y$ is affine, $X$ is also affine. The surface $X$ is Eilenberg–MacLane using Lemma \ref{Lem: Well-known about EM}(3), whereas $Y$ is not, by Lemma~\ref{Lem: Well-known about EM}(4).
	
	It is worth noting, however, that $Y$ is homotopy equivalent to $\PP^1$ by Lemma~\ref{Lem: Homotopy equivalence if the fiber of a bundle is contractible}, and hence homotopy equivalent to $S^2$.
\end{example}

\section*{\bf Acknowledgements} 

The author expresses his sincere gratitude to Prof.~R.~V.~Gurjar for suggesting one of the main problems that inspired this work. He is also deeply thankful to Prof.~Sudarshan~R.~Gurjar for several helpful discussions, for carefully reading the initial draft of the manuscript, and for many valuable suggestions that significantly improved the exposition. The author is grateful to both of them for generously allowing him to publish this manuscript as a single-authored work. He acknowledges financial support from the Department of Science and Technology, Government of India, through the INSPIRE Faculty Fellowship (Reference No.: DST/INSPIRE/04/2024/003379).

\section*{\bf Statements and Declarations}

There is no conflict of interest regarding this manuscript. No external funding was received for it. No data were generated or used.

\bibliographystyle{alpha}
\bibliography{ref}
\end{document}